\documentclass[amssymb, 11pt]{amsart}
\usepackage{latexsym}
\usepackage{young}

\newlength{\standardunitlength}
\newcommand{\bea}{\begin{eqnarray}}
\newcommand{\ena}{\end{eqnarray}}
\newcommand{\beas}{\begin{eqnarray*}}
\newcommand{\enas}{\end{eqnarray*}}

\newcommand{\ignore}[1]{}
\newcommand{\fh}{\mathfrak{h}}

\newcommand{\GL}{\operatorname{GL}}
\newcommand{\GU}{\operatorname{GU}}
\newcommand{\Sp}{\operatorname{Sp}}
\newcommand{\SL}{\operatorname{SL}}

\newcommand{\Mat}{\operatorname{Mat}}
\newcommand{\Com}{\operatorname{Com}}

\newcommand{\fs}{\mathfrak{sp}}
\newcommand{\fo}{\mathfrak{so}}

\newcommand{\rank}{\mathrm{rank \ }}

\newcommand{\Sym}{\mathrm{Sym}}

\newcommand{\F}{\mathbb{F}}

\newtheorem{prop}{Proposition}[section]

\newtheorem{lemma}[prop]{Lemma}

\newtheorem{theorem}[prop]{Theorem}

\begin{document}

\title [Enumeration of Commuting Pairs] {Enumeration of Commuting Pairs in Lie Algebras
over Finite Fields}

\author{Jason Fulman}

\address{Department of Mathematics\\
        University of Southern California\\
        Los Angeles, CA, 90089}
\email{fulman@usc.edu}

\author{Robert Guralnick}
\address{Department of Mathematics\\University of Southern California\\
        Los Angeles, CA, 90089}
\email{guralnic@usc.edu}



\date{October 17, 2016}

\begin{abstract} Feit and Fine derived a generating function for the number of ordered pairs of commuting $n \times n$ matrices
over the finite field $\mathbb{F}_q$. This has been reproved and studied by Bryan and Morrison from the viewpoint of motivic
Donaldson-Thomas theory. In this note we give a new proof of the Feit-Fine result, and generalize it to the Lie algebra of
finite unitary groups and to the Lie algebra of odd characteristic finite symplectic groups. We extract some asymptotic
information from these generating functions. Finally, we derive generating functions for the number of commuting nilpotent
elements for the Lie algebras of the finite general linear and unitary groups, and of odd characteristic symplectic groups.
\end{abstract}

\maketitle

\section{Introduction}

Let $G_0=1$, and for $n \geq 1$, let $G_n$ be the number of ordered pairs of commuting $n \times n$ matrices (not necessarily invertible) over $\mathbb{F}_q$. Feit and Fine \cite{FF} proved the following theorem.

\begin{theorem} \label{ff}
\[ \sum_{n \geq 0} \frac{G_n u^n}{|\GL(n,q)|} = \prod_{i \geq 1} \prod_{l \geq 0} \frac{1}{1-u^i/q^{l-1}} .\]
\end{theorem}

Theorem \ref{ff} was reproved and refined by Bryan and Morrison \cite{BM} from the viewpoint of motivic Donaldson-Thomas theory, using ``power structures''.
See also \cite{BBS}.

In this paper we give a new proof of the Feit-Fine theorem. Answering a question of Matt Young (an expert in Donaldson-Thomas theory), we obtain analogs of the Feit-Fine theorem for unitary groups and for odd characteristic symplectic groups. More precisely, let $U_0=1$ and for $n \geq 1$, let $U_n$ be the number of ordered pairs of commuting elements of the Lie algebra of $\GU(n,q)$. We will show that  \[ \sum_{n \geq 0} \frac{U_n u^n}{|\GU(n,q)|} = \prod_{i \geq 1} \prod_{l \geq 0} \frac{1}{1-(-1)^l u^i/q^{l-1}}   \] Let $S_0=1$ and for $n \geq 1$, let $S_n$ be the number of ordered pairs of commuting elements of the Lie algebra of $\Sp(2n,q)$ in odd characteristic. We will prove that \[ \sum_{n \geq 0} \frac{S_n u^n}{|\Sp(2n,q)|} = \frac{\prod_{i \geq 1} (1+u^i)} {\prod_{i \geq 1} \prod_{l \geq 0} (1- u^i/q^{2l-1})} .\] Our methods will also work for odd characteristic orthogonal groups, but we do not treat them as this is quite tedious and the ideas are the same as for odd characteristic symplectic groups. On the other hand, for various reasons
(e.g.  nilpotent classes are no longer determined by their Jordan form even in the algebraic group and the description of semisimple classes is different
as well), even characteristic symplectic and orthogonal groups are much trickier, and are beyond the scope of this paper.

It is easy to extract asymptotic information from our generating functions. For example from the original Feit-Fine result, we will conclude that for $q$ fixed
and $n$ tending to infinity, \[ \lim_{n \rightarrow \infty} \frac{G_n}{q^{n^2+n}} = \prod_{j \geq 1} (1-1/q^j)^{-j} .\] This is consistent with the fact that the variety of $n \times n$ commuting matrices over an algebraically closed field has dimension $n^2+n$.

Finally, we derive an analog of the Feit-Fine theorem for ${\it nilpotent}$ commuting pairs. Let $NG_0=1$, and for $n \geq 1$, let $NG_n$ be the number of ordered pairs of commuting nilpotent $n \times n$ matrices over $\mathbb{F}_q$. We will prove that
\[ \sum_{n \geq 0} \frac{NG_n u^n}{|\GL(n,q)|} = \prod_{i \geq 1} \prod_{l \geq 1} \frac{1}{1-u^i/q^l} .\] We also obtain generating functions for
$NU_n$ (the number of ordered pairs of commuting nilpotent elements of the Lie algebra of $\GU(n,q)$), and for $NS_n$ (the number of ordered pairs of commuting nilpotent elements of the Lie algebra of the odd characteristic symplectic group $\Sp(2n,q)$).

We note there  has been quite a lot of work on commuting varieties.
One of the first results in this area
is the result of Motzin and Taussky \cite{MT} that the variety of commuting pairs in $\Mat(n,k)$ with $k$
algebraically closed is irreducible of dimension $n^2 + n$.   This was extended by Richardson \cite{R2}
to reductive Lie algebras $\mathfrak{g}$  in characteristic $0$ and by Levy \cite{Le} in good characteristic
with an extra mild assumption.
It follows easily that in this case the dimension of the commuting variety is $\dim \mathfrak{g} + \rank \mathfrak{g}$.
Recall that good characteristic means any characteristic in type $A$, characteristic not $2$ for groups of type $B,C,D$ and 
characteristic not $2$ or $3$
for $G_2, F_4, E_6$ and $E_7$ and characteristic not $2, 3$ or $5$ for $E_8$.    Very good characteristic means good
characteristic for groups not of type $A$ and characteristic not dividing $n+1$ for groups of type $A_n$.

There has also been quite a bit of study for the variety of commuting nilpotent pairs in a simple Lie algebra $\mathfrak{g}$
corresponding to a Dynkin diagram in good characteristic.
The components are in bijection with the set of distinguished orbits of nilpotent elements (i.e. the centralizer
in the group contains no nontrivial torus) and all components have dimension $\dim \mathfrak{g}$
(under the assumption that the characteristic is very good).  This was proved by Premet \cite{Pr} after some
earlier  results of Baranovosky \cite{Ba} and Basili \cite{Bs} for  $\frak{gl}_n$  (with some restrictions on the characteristic).
In particular, this variety is irreducible for $\Mat(n,k)$.  It will have a large number of components for some of the other
simple Lie algebras.     For example when $\frak{g} = \frak{sp}(2n,k)$ and the characteristic of $k$ is not $2$, the number
of irreducible components is the number of partitions of $n$ with all parts of distinct sizes.

Our methods would give similar answers for the special orthogonal Lie algebras in characteristic not $2$.  We do not work out the details
of this case and leave it to the interested reader.

The organization of this paper is as follows. Section \ref{GL} treats the general linear groups. It gives a new proof of the Feit-Fine theorem, obtains asymptotic information from it, and enumerates commuting pairs of nilpotent matrices. Section \ref{U} studies the Lie algebra of the finite unitary groups $\GU(n,q)$, giving unitary analogs of all the results in Section \ref{GL}. Section \ref{Sp} studies the Lie algebra of the finite symplectic groups $\Sp(2n,q)$ in the case of odd characteristic, giving symplectic analogs of all the results in Section \ref{GL}.

\section{General linear groups} \label{GL}

We use the following notation about partitions. Let
$\lambda$ be a partition of some non-negative integer $|\lambda|$ into
parts $\lambda_1 \geq \lambda_2 \geq \cdots$. Let $m_i(\lambda)$ be the
number of parts of $\lambda$ of size $i$, and let $\lambda'$ be the
partition dual to $\lambda$ in the sense that $\lambda_i' = m_i(\lambda) +
m_{i+1}(\lambda) + \cdots$. It is also useful to define the diagram
associated to $\lambda$ by placing $\lambda_i$ boxes in the $i$th row. We
use the convention that the row index $i$ increases as one goes downward.
So the diagram of the partition $(5441)$ is

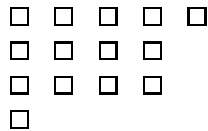
\[ \begin{array}{c c c c c} \framebox{} & \framebox{} & \framebox{}
& \framebox{} &  \framebox{} \\ \framebox{} &  \framebox{}& \framebox{} &
\framebox{} & \\
\framebox{} & \framebox{} & \framebox{} & \framebox{}&    \\ \framebox{} &
& & &
\end{array}
\] and $\lambda_i'$ can be interpreted as the size of the $i$th column.

The notation $(1/q)_r$ will denote $(1-1/q)(1-1/q^2) \cdots (1-1/q^r)$.

The next two lemmas are well known and will be useful. Lemma \ref{euler} is due
to Euler and is proved on page 19 of \cite{A}.

\begin{lemma} \label{euler}
\[ \sum_{m \geq 0} \frac{u^m}{(1/q)_m} = \prod_{l \geq 0} \frac{1}{1-u/q^l} .\]
\end{lemma}

Lemma \ref{poly} is elementary. In its statement, and for the rest of the paper,
we let $d(\phi)$ denote the degree of a polynomial $\phi$. For completeness, we include a proof.

\begin{lemma} \label{poly}
\[ \prod_{\phi} (1-u^{d(\phi)}) = 1 - uq, \] where the product is over all monic, irreducible
polynomials over the finite field $\mathbb{F}_q$.
\end{lemma}

\begin{proof} By unique factorization in $\mathbb{F}_q[x]$, the coefficient of $u^n$ in the reciprocal of the
left hand side is the number of degree $n$ monic polynomials with coefficients in $\mathbb{F}_q$. This is
equal to $q^n$, which is the coefficient of $u^n$ in $(1-uq)^{-1}$.
\end{proof}

We let $\Mat(n,q)$ denote the set of $n \times n$ matrices (not necessarily invertible) over the
field $\mathbb{F}_q$. Recall the rational canonical form of an element of $\Mat(n,q)$. This
associates to each monic, irreducible polynomial $\phi$ over $\mathbb{F}_q$ a partition $\lambda_{\phi}$
such that \[ \sum_{\phi} d(\phi) |\lambda_{\phi}| = n .\] For further background on rational
canonical forms, one can consult Chapter 6 of \cite{H}.

Lemma \ref{jordan} is proved in Stong \cite{S}, and calculates the number of elements of $\Mat(n,q)$
with given rational canonical form.

\begin{lemma} \label{jordan} Suppose that
\[ \sum_{\phi} d(\phi) |\lambda_{\phi}| = n, \] so that $\{ \lambda_{\phi} \}$ is a possible rational
canonical form of an element of $\Mat(n,q)$. Then the number of elements of $\Mat(n,q)$ with rational
canonical form  $\{ \lambda_{\phi} \}$ is equal to
\[ \frac{|\GL(n,q)|}{\prod_{\phi} q^{d(\phi) \sum_i (\lambda_{\phi,i}')^2}
\prod_i (1/q^{d(\phi)})_{m_i(\lambda_{\phi})}} .\]
(Recall that $\lambda_{\phi,i}'$ is the size of the $i$th column of the partition $\lambda_{\phi}$).
\end{lemma}

Lemma \ref{comwith} counts the number of elements of $\Mat(n,q)$ which commute with a given element
of $\Mat(n,q)$ which has rational canonical form $\{ \lambda_{\phi} \}$.

\begin{lemma} \label{comwith} Let $x$ be an element of $\Mat(n,q)$ with rational canonical form
$\{ \lambda_{\phi} \}$. Then the number of elements of $\Mat(n,q)$ which commute with $x$ is equal to
\[ \prod_{\phi} q^{d(\phi) \sum_i (\lambda_{\phi,i}')^2} .\]
\end{lemma}

\begin{proof}   Note that the exponent in the statement of the result is precisely the dimension of the centralizer
of $x$ (either in $\Mat(n,q)$ or $\Mat(n,k)$ where $k$ is the algebraic closure  of our finite field).    Note that the
centralizer breaks up into the product of the centralizers corresponding to each irreducible polynomial $\phi$.
By passing to a splitting field for $\phi$, it suffices to assume that $\phi$ has degree $1$ and indeed
that $\phi = x$.

Now we can decompose the underlying space $V$ into a direct sum $V_1 \oplus \ldots \oplus V_m$
where $x$ acts via a single Jordan block on $V_i$ for size $\lambda_i$.   Then the dimension of
the centralizer is the sum of the dimensions of of $Hom_{k[x]}(V_i, V_j)$.   Since $V_j$ and $V_j$
are cyclic $k[x]$-modules, this dimension is  the minimum of $\lambda_i, \lambda_j$. Clearly
\[ \sum_{i,j \geq 1} min(\lambda_i,\lambda_j) =
2 \sum_{k<l} k m_k(\lambda) m_l(\lambda) + \sum_i i m_i(\lambda)^2 .\] This is equal to
$\sum_i (\lambda_i')^2$, as can be seen by writing
\[ \lambda_i' = m_i(\lambda) + m_{i+1}(\lambda) + \cdots .\]
\end{proof}

Now, we reprove the Feit-Fine theorem.

\begin{proof} We let $[u^n] f(u)$ denote the coefficient of $u^n$ in an expression $f(u)$.
It follows from Lemmas \ref{jordan} and \ref{comwith} that $G_n$ (the number of
ordered pairs of commuting elements of $\Mat(n,q)$) is equal to
\[ |\GL(n,q)| [u^n] \prod_{\phi} \sum_{\lambda} \frac{u^{|\lambda| d(\phi)}}
{\prod_i (1/q^{d(\phi)})_{m_i(\lambda)}}, \] where the product
is over all monic irreducible polynomials $\phi$ over $\mathbb{F}_q$. This is equal to
\[ |\GL(n,q)| [u^n] \prod_{\phi} \prod_{i \geq 1} \sum_{m \geq 0} \frac{u^{im \cdot d(\phi)}}{(1/q^{d(\phi)})_m} .\]
By Lemma \ref{euler}, this is equal to
\[ |\GL(n,q)| [u^n] \prod_{\phi} \prod_{i \geq 1} \prod_{l \geq 0} \frac{1}{1-u^{i \cdot d(\phi)}/q^{l \cdot d(\phi)}} .\]
Switching the order of the products and applying Lemma \ref{poly} shows that this is equal to
\begin{eqnarray*}
& & |\GL(n,q)| [u^n] \prod_{i \geq 1} \prod_{l \geq 0} \prod_{\phi} \frac{1}{1-u^{i \cdot d(\phi)}/q^{l \cdot d(\phi)}} \\
& = & |\GL(n,q)| [u^n] \prod_{i \geq 1} \prod_{l \geq 0} \frac{1}{1-u^i/q^{l-1}}.
\end{eqnarray*} This completes the proof.
\end{proof}

Next we derive asymptotic information from the Feit-Fine generating function.

Given a power series $f(u)$, let $[u^n] f(u)$ denote the coefficient of $u^n$ in $f(u)$.
The following lemma will be useful.

\begin{lemma} \label{tay} If the Taylor series of $f$ around $0$ converges at $u=1$, then \[ \lim_{n \rightarrow \infty} [u^n] \frac{f(u)}{1-u} = f(1) .\]
\end{lemma}

\begin{proof} Write the Taylor expansion $f(u)=\sum_{n=0}^{\infty} a_n u^n$. Then observe that $[u^n] \frac{f(u)}{1-u} = \sum_{i=0}^n a_i$.
\end{proof}

\begin{theorem} For $q$ fixed,
\[ \lim_{n \rightarrow \infty} \frac{G_n}{q^{n^2+n}} = \prod_{j \geq 1} (1-1/q^j)^{-j} .\]
\end{theorem}

\begin{proof} We know that
\begin{eqnarray*}
G_n & = & |\GL(n,q)| [u^n] \prod_{i \geq 1} \prod_{l \geq 0} \frac{1}{1-u^i/q^{l-1}} \\
& = & |\GL(n,q)| q^n [u^n] \prod_{i \geq 1} \prod_{l \geq 0} \frac{1}{1-u^i/q^{i+l-1}} \\
& = & q^{n^2+n} (1/q)_n [u^n] \prod_{i \geq 1} \prod_{l \geq 0} \frac{1}{1-u^i/q^{i+l-1}}.
\end{eqnarray*}

Thus \[ \lim_{n \rightarrow \infty} \frac{G_n}{q^{n^2+n}} = \prod_{j \geq 1} (1-1/q^j)
\lim_{n \rightarrow \infty} [u^n] \prod_{i \geq 1} \prod_{l \geq 0} \frac{1}{1-u^i/q^{i+l-1}} .\] Since the $i=1,l=0$
case of $1/(1-u^i/q^{i+l-1})$ is equal to $1/(1-u)$, it follows from Lemma \ref{tay}
that \[ \lim_{n \rightarrow \infty} [u^n] \prod_{i \geq 1} \prod_{l \geq 0} \frac{1}{1-u^i/q^{i+l-1}}
= \prod_{j \geq 1} (1-1/q^j)^{-j-1}, \] which implies the theorem.
\end{proof}

{\it Remark:} Let $\Com(G)$ denote the number of ordered pairs of commuting elements of a finite group $G$. It is
easy to see that $\Com(G)$ is equal to $k(G)|G|$, where $k(G)$ is the number of conjugacy classes of $G$. From
\cite{FG}, $\lim_{n \rightarrow \infty} \frac{k(\GL(n,q))}{q^n} = 1$. Thus,
\[ \lim_{n \rightarrow \infty} \frac{\Com(\GL(n,q))}{q^{n^2+n}} = \lim_{n \rightarrow \infty} \frac{|\GL(n,q)|}{q^{n^2}} =
\prod_{j \geq 1} (1-1/q^j) .\]

To close this section, we enumerate pairs of commuting nilpotent $n \times n$ matrices. Let $NG_0=1$, and for $n \geq 1$, let $NG_n$ be the number of ordered pairs of commuting $n \times n$ nilpotent matrices over $\mathbb{F}_q$.

\begin{theorem} \label{nilGL}
\[ \sum_{n \geq 0} \frac{NG_n u^n}{|\GL(n,q)|} = \prod_{i \geq 1} \prod_{l \geq 1} \frac{1}{1-u^i/q^l} .\]
\end{theorem}

The next result gives the number of nilpotent elements commuting with a given nilpotent element.

\begin{lemma} \label{countnilcent}
The number of nilpotent $n \times n$ matrices that commute with a fixed nilpotent matrix of type $\lambda$ is equal to
\[ q^{\sum_i (\lambda_i')^2 - \sum_i m_i(\lambda)} .\]
\end{lemma}

\begin{proof}    It is well known and easy that the centralizer in $\Mat(n,q)$ is
$$
\Mat(m_1,q) \times \ldots \times \Mat(m_r,q),
$$
modulo its nilpotent radical. From \cite{Lh}, the number of nilpotent matrices in $\Mat(m,q)$ is  $q^{m^2-m}$.
The result follows.
\end{proof}

Now we prove Theorem \ref{nilGL}.

\begin{proof} (Of Theorem \ref{nilGL}) Lemma \ref{jordan} gives that the number of $n \times n$ nilpotent
matrices of Jordan type $\lambda$ is equal to \[ \frac{|\GL(n,q)|}{q^{\sum_i (\lambda_i')^2} \prod_i (1/q)_{m_i(\lambda)}} .\]
Combining this with Lemma \ref{countnilcent} gives that
\[ NG_n = |\GL(n,q)| \sum_{|\lambda|=n} \frac{1}{\prod_i (1/q)_{m_i(\lambda)} q^{m_i(\lambda)}} .\]

Thus
\[ NG_n = |\GL(n,q)| [u^n] \prod_{i \geq 1} \sum_{m \geq 0} \frac{u^{im}}{(1/q)_m q^m}. \]
Now replacing $u$ by $u^i/q$ in Lemma \ref{euler} gives that
\[ \sum_{m \geq 0} \frac{u^{im}}{(1/q)_m q^m} = \prod_{l \geq 1} \frac{1}{1-u^i/q^l} .\]
It follows that
\[ NG_n = |\GL(n,q)| [u^n] \prod_{i \geq 1} \prod_{l \geq 1} \frac{1}{1-u^i/q^l}, \]
proving the theorem.
\end{proof}

\section{Unitary groups} \label{U}

We first prove some analogs of the $\GL$ case.

Let $F=\mathbb{F}_{q^2}$ and let $y^q= \bar{y}$ denote the $q$-Frobenius map on $F$.
Let $F_0 = \mathbb{F}_q$.   Let $k$ denote the algebraic
closure of $F$.
We extend this map to an $F_0$ automorphism of $\Mat(n,q^2)$.    Let $A^{\top}$ denote the transpose of $A$.

Since $\GU(n,q)$ are those matrices $U$  (in $\Mat(n,q^2)$) satisfying  $U{\bar{U}}^{\top} = I$, the Lie algebra of $\GU(n,q)$ are those matrices $A \in \Mat(n,q^2)$ such
that ${\bar{A}}^{\top} = - A$.  We call these matrices skew Hermitian.   We will occasionally write $\GL^{\epsilon}(n,q)$ for $\GL(n,q)$ with $\epsilon=+$
and $\GU(n,q)$ with $\epsilon=-$.

  Choose $\theta \in F$ so that $\theta \bar{\theta} = -1$ (if $F$ has characteristic $2$,
then take $\theta =1$ and if $F$ has odd characteristic, just take $\theta$ so that $\theta^2 \in F_0$ is a nonsquare).   Then $A$ skew Hermitian
implies that $B:=\theta A$ is Hermitian (i.e. ${\bar{B}}^{\top} = B$).  Let $\mathfrak{h}$ denote the set of Hermitian matrices.  Note that the orbits of $\GU(n,q)$ on
$\fh$
are in bijection with those of $\theta \fh$ and of course the set of commuting pairs are in bijection as well.  Thus, we will obtain counts
for the set of commuting pairs in $\fh$ (note that $\fh$ is a Jordan algebra).

We first observe that the orbits of $\GU(n,q)$ on $\fh$ are in bijection (which is  not size preserving) with those of $\GL(n,q)$
on $\Mat(n,q)$.

\begin{lemma}  Let $A \in \Mat(n,q^2)$.  Then $A$ is similar
 to an element of $\fh$ if and only if    $A$ is similar to an element of $\Mat(n,q)$.
 Moreover,  $A, B \in \fh$ are conjugate in $\GL(n,k)$ if and only if they are conjugate by an element of  $\GU(n,q)$.
 \end{lemma}

 \begin{proof}
 First suppose that $A \in \fh$.   Since $\bar{A}^{\top} = A$ and $A$ and $A^{\top}$ are similar, it follows that $A$ and $\bar A$
 are similar.   Thus all invariant factors of $A$ are defined over $F_0$, whence $A$ is similar to an element of $\Mat(n,q)$.

 Conversely, assume that $A \in \Mat(n,q)$.    Since all nondegenerate Hermitian forms of dimension $n$ over $F$ are equivalent,
 it suffices to show that $A$ is self adjoint with respect to some Hermitian form -- i.e. $AH=H\bar{A}^{\top}$ for some invertible
 Hermitian matrix $H$.
  By \cite{TZ},  $AS = SA^{\top}$ for some symmetric invertible matrix $S \in \GL(n,q)$.  Of course, $S$
 is Hermitian and  $A$ is self adjoint with respect to the inner product determined by $S$, whence the first statement holds.

 We now prove the last statement.  Of course, if $A, B$ are conjugate via an element of $\GU(n,q)$ they are conjugate in $\GL(n,k)$.
 The other direction follows by the Lang-Steinberg theorem.  If $B = UAU^{-1}$, then $A, B \in \fh$ implies that
 $UAU^{-1} = \bar{U}^{-\top}A\bar{U}^{\top}$.   Thus,  $\bar{U}^{\top}U$ centralizes $A$.  By the Lang-Steinberg theorem (since
 the centralizer of $A$ in $\GL(n,k)$ is connected),  $\bar{U}^{\top}U  = \bar{X}^{\top}X$ for some $X$ centralizing $A$.
 Thus,  $UX^{-1}  A  X U^{-1}  = B$ and $UX^{-1} \in \GU(n,q)$.
 \end{proof}

 \begin{lemma}  Let $A \in \fh$.   Then the dimension of the centralizer (over $F_0$) of $A$ in $\fh$ is the same as
 the dimension of the centralizer (over $F$) of $A$ in $\Mat(n,q^2)$ (which is the same as the dimension of the centralizer
 in $\Mat(n,k)$).
 \end{lemma}

 \begin{proof}  We can write  $\Mat(n,q^2) = \fh \oplus  \alpha \fh$ (as $F_0$ spaces) for any element $\alpha \in F \setminus{F_0}$.
 Note that the map $X \rightarrow [A,X]$ preserves this decomposition.  Thus, $C$ the centralizer of $A$ is a direct sum of its intersections
 with $\fh$ and $\alpha \fh$.   Clearly,  $\dim_{F_0}  C \cap \fh = \dim_{F_0} \alpha \fh$.    Letting $kC$ denote
 the centralizer of $A$ in $\Mat(n,k)$, we  thus obtain  $\dim_{F_0} C \cap \fh =
 \dim_F  C = \dim_k  kC$.
 \end{proof}

 \begin{lemma}  Let $A \in \fh$.     Let $f(x) \in F_0[x]$ denote the characteristic polynomial of $A$ and
 factor $f(x) = \prod_i \phi_i(x)^{e_i}$ where the $\phi_i$ are distinct monic irreducible polynomials in $F_0[x]$.
   Let $V$ denote the column space of dimension $n$   over $F$.  Then
 $V$ is an orthogonal direct sum of spaces $V_i$ such that the characteristic polynomial of $A$ on $V_i$
 is a power of $\phi_i$ with $\phi_i \ne \phi_j$.
 \end{lemma}

 \begin{proof}  By the previous lemma,  for any given similarity class, there is a representative that is of the form
 given.  Also by the previous lemma, the similarity class determines the orbit under the unitary group and the
 result follows.
 \end{proof}

 We now want to determine the centralizer of $A \in \fh$.  By the previous lemma, it suffices to do this under the
 assumption that the characteristic polynomial of $A$ is $\phi(x)^m$ where $\phi(x) \in F_0[x]$ is irreducible of degree $d$.
 The similarity class of such an $A$ (given $\phi(x)$) is determined by a partition of $m$.
 So write $m = \sum_i e_i \lambda_i$ where $\lambda_1 > \lambda_2 > \ldots$ and the partition has $e_i$ parts of length
 $\lambda_i$.

 \begin{lemma}   Let $A \in \fh$.  Let $G=\GU(n,q)$.  Then $C_G(A)$ has unipotent radical $Q$ of order $q^{\ell}$ where
 $\ell$ is the dimension of the unipotent radical of  $C_{\GL(n,k)}(A)$ and
 $$
 C_G(A)/Q  \cong  \prod_i  \GL^{\epsilon}(e_i, q^d),
 $$
 where $\epsilon = +$ if $d$ is even and $-$ if $d$ is odd.
 \end{lemma}

 \begin{proof}   Let $X$ be the centralizer of $A$ in $\GL(n,k)$.   This has a unipotent radical $U$ and
 $Q$ is just the fixed points of $U$ under the endomorphism of $\GL(n,k)$ defining $GL^{\epsilon}(n,q)$.
 For any Steinberg-Lang endomorphism on a connected unipotent group, the size of the fixed point group
 is as given.

 So we now focus on the reductive part of the centralizer.   Let $V$ denote the natural module (over $F$).

 First suppose that $d$ is even.  Then $\phi(x) = \phi_1(x) \phi_2(x)$ where $\phi_i(x)$ has degree $d/2$ and
 $\phi_2(x)$ is the Galois conjugate of $\phi_1(x)$ under the $q$ power map.

 It follows that $V = V_1 \oplus V_2$ where $V_i$ is the kernel of $\phi_i(A)^n$.  Then $V_1$ and $V_2$ are totally
 singular spaces (with respect to the Hermitian form) and $A$ has characteristic polynomial a power of
 $\phi_i(x)$ on $V_i$.

 Let $A_i$ denote the restriction of $A$ to $V_i$.
 Clearly the centralizer of $A$ preserves each $V_i$.  Note that if $g \in \GU(n,q)$ fixes each $V_i$, then
 with respect to an appropriate basis  if $g_i$ denotes $g$ restricted to $V_i$, then $g_2 = \bar{g}_1^{-\top}$.
 The subgroup stabilizing $V_1$ and $V_2$ is isomorphic to $\GL(n/2, q^2)$ and the
 centralizer of $A$ in $U$ is isomorphic
 to the centralizer of $A_1$ in $\GL(n,q^2)$.  The result then follows by the result for $\GL$.

 Now suppose that $d$ is odd.   Then $\phi(x)$ remains irreducible over $F$.   By the uniqueness of the canonical forms,
 we may assume that $V$ is an orthogonal direct sum of subspaces $V_i$ such that on $V_i$ the rational canonical
 form has all blocks the same (i.e. the minimal  and characteristic polynomials are both powers  of  $\phi(x)$).

 By considering the $\GL$ case (over $k$) we see that the reductive part of the centralizer (up to conjugacy) preserves
 each $V_i$.   Thus, it suffices to assume that $V = V_i$.  In the notation given above, this means that $\lambda_1=\lambda$,
 $e_1=e$ and $de\lambda = n$.

 First suppose that
 $\phi(x)$ is the minimal polynomial of $A$ (i.e. $A$ is a semisimple element and $\lambda  =1$).     Then we can decompose the
 space as a direct sum of $e$ orthogonal spaces each of dimension $d$ so that $A$ acts irreducibly on each space.  By the $\GL$ case,
 the centralizer in $\GL(n,F)$ is just $\GL(e,q^{2d})$.   We can view the standard Hermitian form as inducing one on the $e$ dimensional vector
 space over $\F_{q^{2d}}$ and clearly  $\GL(e,q^{2d}) \cap \GU(de,q) = \GU(e,q)$ as claimed.

 Finally suppose that $\lambda > 1$.   Then we can decompose the space as $V_1 \oplus \ldots \oplus V_{\lambda}$ (this is not an
 orthogonal decomposition and indeed $V_1$ will be totally singular) so that the matrix of $A$ with respect to this decomposition is

 $$
 \begin{pmatrix}
 B & I & 0 & \ldots  &0 \\
 0  & B &  I & \ldots &  0 \\
&  & \cdots & &\\
 0 & 0 & 0 & \ldots & B \\
  \end{pmatrix}.
 $$
 \bigskip

 Here $B$ is semisimple and has minimal polynomial $\phi(x)$ and characteristic polynomial $\phi(x)^e$.
 We can take the Hermitian form preserved by this matrix to be in block diagonal form (as in the above
 equation) with the antidiagonal blocks all being $I$ and all other blocks being $0$.   So we need to
 compute the centralizer with respect to the unitary group preserving this form and as before we only need
 to compute the reductive part of this centralizer.    Again, considering the case of $\GL$, we see that the
 reductive part of the centralizer will be block diagonal.    It is straightforward to see that if $x$ is block
 diagonal centralizing $A$, then all diagonal blocks are the same.    By the previous case this implies
 that $x \in \GU(e,q^d)$ and clearly any such element commutes with $A$, whence
 the result.
\end{proof}

Lemma \ref{jordanu} is a unitary analog of Lemma \ref{jordan} and the proof is identical given the
previous results.

\begin{lemma} \label{jordanu} The orbits of $\GU(n,q)$ on its Lie algebra correspond to $\{ \lambda_{\phi} \}$
 satisfying
\[ \sum_{\phi} d(\phi) |\lambda_{\phi}| = n ,\] that is to the possible rational canonical forms of
an element of $\Mat(n,q)$. The size of the orbit corresponding to data $\{ \lambda_{\phi} \}$ is
\[ \frac{|\GU(n,q)|}{\prod_{\phi} C_{\phi}(\lambda)} .\]
Here
\[ C_{\phi}(\lambda) = q^{d(\phi) \sum_i (\lambda_{\phi,i}')^2}
\prod_i (1/q^{d(\phi)})_{m_i(\lambda_{\phi})} \] if $d(\phi)$ is even, and
\[ C_{\phi}(\lambda) = q^{d(\phi) \sum_i (\lambda_{\phi,i}')^2}
\prod_i (1/q^{d(\phi)})_{m_i(\lambda_{\phi}),q \mapsto -q} \] if $d(\phi)$ is odd.
Here the notation $(1/q^{d(\phi)})_{m_i(\lambda_{\phi}),q \mapsto -q}$ means that
we replace $q$ by $-q$ in the expression $(1/q^{d(\phi)})_{m_i(\lambda_{\phi})}$.
\end{lemma}

Lemma \ref{comwithu} is a unitary analog of Lemma \ref{comwith}, and the proof is identical
given the previous results.

\begin{lemma} \label{comwithu} Let $x$ be an element of the Lie algebra of $\GU(n,q)$, with data
$\{ \lambda_{\phi} \}$. Then the number of elements of the Lie algebra of $\GU(n,q)$ which commute with
$x$ is equal to
\[ \prod_{\phi} q^{d(\phi) \sum_i (\lambda_{\phi,i}')^2} .\]
\end{lemma}

Now we prove the unitary analog of the Feit-Fine theorem. Let $U_0=1$ and for $n \geq 1$, let $U_n$
be the number of ordered pairs of commuting elements of the Lie algebra of $\GU(n,q)$.

\begin{theorem}
\[ \sum_{n \geq 0} \frac{U_n u^n}{|\GU(n,q)|} = \prod_{i \geq 1} \prod_{l \geq 0} \frac{1}{1-(-1)^l u^i/q^{l-1}} \]
\end{theorem}

\begin{proof} Let $[u^n] f(u)$ denote the coefficient of $u^n$ in an expression $f(u)$.
It follows from Lemmas \ref{jordanu} and \ref{comwithu} that $U_n$ is equal to
$|\GU(n,q)|$ multiplied by the coefficient of $u^n$ in
\[ \prod_{\phi \atop d(\phi) \ even} \sum_{\lambda} \frac{u^{|\lambda| d(\phi)}}
{\prod_i (1/q^{d(\phi)})_{m_i (\lambda)}}
\prod_{\phi \atop d(\phi) \ odd} \sum_{\lambda} \frac{u^{|\lambda| d(\phi)}}
{\prod_i (1/q^{d(\phi)})_{m_i (\lambda), q \mapsto -q}} .\] This is equal to
$|\GU(n,q)|$ multiplied by the coefficient of $u^n$ in
\[ \prod_{\phi \atop d(\phi) \ even} \prod_{i \geq 1} \sum_{m \geq 0}
\frac{u^{im \cdot d(\phi)}}{(1/q^{d(\phi)})_m}
\prod_{\phi \atop d(\phi) \ odd} \prod_i \sum_{m \geq 0} \frac{u^{im \cdot d(\phi)}}
{\prod_i (1/q^{d(\phi)})_{m, q \mapsto -q}} .\]

By Lemma \ref{euler}, this is equal to
\[ |\GU(n,q)| [u^n] \prod_{\phi} \prod_{i \geq 1} \prod_{l \geq 0}
\frac{1}{1-(u^i (-1)^l/q^l)^{d(\phi)}} .\] Switching the order of the
products and applying Lemma \ref{poly} shows that this is equal to
\begin{eqnarray*}
& &  |\GU(n,q)| [u^n] \prod_{i \geq 1} \prod_{l \geq 0} \prod_{\phi}
\frac{1}{1-(u^i (-1)^l/q^l)^{d(\phi)}} \\
& = & |\GU(n,q)| [u^n] \prod_{i \geq 1} \prod_{l \geq 0}
\frac{1}{1-(-1)^l u^i/q^{l-1}}.
\end{eqnarray*} This completes the proof.
\end{proof}

Theorem \ref{asymU} determines the asymptotics of the sequence $U_n$.

\begin{theorem} \label{asymU} For $q$ fixed,
\[ \lim_{n \rightarrow \infty} \frac{U_n}{q^{n^2+n}} = \prod_{i \ odd} (1-1/q^i)^{-1}
\prod_{i \ even} (1-1/q^i)^{- \lfloor i/4 \rfloor} .\]
\end{theorem}

\begin{proof} We know that
\begin{eqnarray*}
U_n & = & |\GU(n,q)| [u^n] \prod_{i \geq 1} \prod_{l \geq 0} \frac{1}{1-(-1)^l u^i/q^{l-1}} \\
& = & |\GU(n,q)| q^n [u^n] \prod_{i \geq 1} \prod_{l \geq 0} \frac{1}{1-(-1)^l u^i/q^{i+l-1}}.
\end{eqnarray*}

Thus \[ \lim_{n \rightarrow \infty} \frac{U_n}{q^{n^2+n}} = \prod_{j \geq 1} (1-(-1)^j/q^j)
\lim_{n \rightarrow \infty} [u^n] \prod_{i \geq 1} \prod_{l \geq 0} \frac{1}{1-(-1)^l u^i/q^{i+l-1}} .\]
Since the $i=1,l=0$ case of $1/(1-(-1)^l u^i/q^{i+l-1})$ is equal to $1/(1-u)$, it follows from
Lemma \ref{tay} that
\begin{eqnarray*}
\lim_{n \rightarrow \infty} \frac{U_n}{q^{n^2+n}} & = & \frac{1}{(1-1/q)} \cdot \frac{1}{(1+1/q^2)(1-1/q^2)} \\
& & \cdot \frac{1}{(1-1/q^3)(1+1/q^3)(1-1/q^3)} \\
& & \cdot \frac{1}{(1+1/q^4)(1-1/q^4)(1+1/q^4)(1-1/q^4)} \cdots \\
& = & \prod_{i \ odd} (1-1/q^i)^{-1} \cdot \frac{1}{(1+1/q^2)(1-1/q^2)} \\
& & \cdot \frac{1}{(1+1/q^3)(1-1/q^3)} \\
& & \cdot \frac{1}{(1+1/q^4)(1-1/q^4)(1+1/q^4)(1-1/q^4)} \cdots.
\end{eqnarray*} Writing $(1+1/q^k)(1-1/q^k)=(1-1/q^{2k})$ shows that this is equal to
\[ \prod_{i \ odd} (1-1/q^i)^{-1}
\prod_{i \ even} (1-1/q^i)^{- \lfloor i/4 \rfloor}. \]
\end{proof}

{\it Remark:} Let $\Com(G)$ denote the number of ordered pairs of commuting elements of a finite group $G$. As
mentioned earlier, $\Com(G)$ is equal to $k(G)|G|$, where $k(G)$ is the number of conjugacy classes of $G$. From
\cite{FG},
\[ \lim_{n \rightarrow \infty} \frac{k(\GU(n,q))}{q^n} = \prod_{i \geq 1} \frac{(1+1/q^i)}{(1-1/q^i)}.\] Thus,
\begin{eqnarray*}
\lim_{n \rightarrow \infty} \frac{\Com(\GU(n,q))}{q^{n^2+n}} & = & \prod_{i \geq 1} \frac{(1+1/q^i)}{(1-1/q^i)}
\lim_{n \rightarrow \infty} \frac{|\GU(n,q)|}{q^{n^2}} \\
& = & \prod_{i \ odd} \frac{(1+1/q^i)}{(1-1/q^i)} \prod_{i \geq 1} (1+1/q^i).
\end{eqnarray*}

Next we enumerate pairs of commuting nilpotent matrices in the Lie algebra of $\GU(n,q)$. Let $NU_0=1$, and
for $n \geq 1$, let $NU_n$ be the number of ordered pairs of commuting $n \times n$ nilpotent matrices in the Lie
algebra of $\GU(n,q)$.

\begin{theorem} \label{nilU}
\[ \sum_{n \geq 0} \frac{NU_n u^n}{|\GU(n,q)|} = \prod_{i \geq 1} \prod_{l \geq 1} \frac{1}{(1+(-1)^l u^i/q^l)}. \]
\end{theorem}

We next state a lemma about counting nilpotent elements in a Lie subalgebra.

\begin{lemma} \label{lem:nilpotent}   Let $G$ be a connected algebraic group with Lie algebra $\frak{g}$ over an algebraically
closed field of characteristic $p$.   Let $F$ be a Lang-Steinberg endomorphism of $G$ (so $G^F$ is a finite group).
The number of nilpotent elements in $(\mathrm{Lie}(G))^F$ is  $q^{\dim G  - \rank G}$ (where $q=q_F$ is the size of the field
associated with $F$).
\end{lemma}

\begin{proof}   Let $Q$ be the unipotent radical of $G$.  Of course $Q$ is $F$-invariant.   Let $\frak{h}$ be the Lie algebra
of $Q$. Then $\frak{h}$ consists of nilpotent elements and $\frak{h}^F$ will have cardinality  $q^{\dim Q}$.
Thus, we can pass to the case that $G$ is reductive.   In that case, we just apply a result of Lehrer \cite[Cor. 1.15]{Le} (this
also follows by a result of Steinberg \cite{St} in good characteristic -- Steinberg proves the corresponding result for unipotent
elements).
\end{proof}

\begin{lemma} \label{countnilcentU}
The number of nilpotent $n \times n$ matrices which are in the Lie algebra of $\GU(n,q)$ and which commute with a fixed nilpotent matrix
of type $\lambda$ in the Lie algebra of $\GU(n,q)$ is equal to
\[ q^{\sum_i (\lambda_i')^2 - \sum_i m_i(\lambda)} .\]
\end{lemma}

\begin{proof}    Let $A \in \frak{h}$ be nilpotent of type $\lambda$.    As we have already noted,  the dimension of the centralizer of $A$ is
$ \sum_i (\lambda_i')^2$.    We have also seen that the reductive part of the centralizer of $A$
is a direct product of $\GL(m_i(\lambda_i))$, whence    the rank of the centralizer is $\sum_i m_i(\lambda)$.  Thus,
the result follows by Lemma \ref{lem:nilpotent}.
\end{proof}

Now we prove Theorem \ref{nilU}.

\begin{proof} (Of Theorem \ref{nilU}) From Lemma \ref{jordanu} we know that the total number of elements
of the Lie algebra of $\GU(n,q)$ which are nilpotent of type $\lambda$ is equal to
\[ \frac{|\GU(n,q)|}{q^{\sum (\lambda_i')^2} \prod_i (1/q)_{m_i(\lambda),q \mapsto -q}}.\]
Here the notation $(1/q)_{m_i(\lambda),q \mapsto -q}$ means that we replace $q$ by $-q$ in
the expression $(1/q)_{m_i(\lambda)}$.

Combining this with Lemma \ref{countnilcentU} gives that
\begin{eqnarray*}
NU_n & = & |\GU(n,q)| \sum_{|\lambda|=n} \frac{1}{\prod_i q^{m_i(\lambda)} (1/q)_{m_i(\lambda),q \mapsto -q}} \\
& = & |\GU(n,q)| [u^n] \prod_{i \geq 1} \sum_{m \geq 0} \frac{u^{im}}{q^m (1/q)_{m,q \mapsto -q}} \\
& = & |\GU(n,q)| [u^n] \prod_{i \geq 1} \sum_{m \geq 0} \frac{[(-1) u^i]^m}{[q^m (1/q)_m]_{q \mapsto -q}}.
\end{eqnarray*} It follows from Lemma \ref{euler} that
\[ \sum_{m \geq 0} \frac{[(-1) u^i]^m}{q^m (1/q)_m} = \prod_{l \geq 1} \frac{1}{1+u^i/q^l}.\] Replacing $q$ by $-q$
gives that
\[ \sum_{m \geq 0} \frac{[(-1) u^i]^m}{[q^m (1/q)_m]_{q \mapsto -q}} = \prod_{l \geq 1} \frac{1}{1+(-1)^l u^i/q^l}.\]

Hence \[ NU_n = |\GU(n,q)| [u^n] \prod_{i \geq 1} \prod_{l \geq 1} \frac{1}{(1+(-1)^l u^i/q^l)} .\]
\end{proof}

\section{Symplectic Groups} \label{Sp}

Here we consider the symplectic case.  There are several extra difficulties in characteristic $2$.
So let $F=\F_q$ with $q$ an odd prime power.  Let $G=\Sp(2n,q)$.   Let $\fs$ be the Lie algebra of $G$.   This can be viewed in a number
of ways.

One way is as follows.  Let $J$ denote the skew symmetric matrix defining $G$ (i.e. $G$ is the stabilizer of $J$).   Then

$$
\fs := \{ A  \in \Mat(n,q) | AJ = (AJ)^{\top} \}.
$$

We can also view $\fs$ as the linear transformations on the natural module $V$ for $G$ which are self adjoint
with respect to the alternating form defining $G$.   Note that $V$ is a self-dual $G$-module and so $V \cong V^*$
(where $V^*$ denotes the dual module).

First let us make some remarks if we work over the algebraic closure $k$ of $F$.   Note that if $x \in \fs$, then
if $0 \ne \alpha$ is an eigenvalue of $x$, so is $-\alpha$ and the generalized $\alpha$ eigenspace of $x$
is totally singular.   The sum of the generalized $\alpha$ and $-\alpha$ eigenspaces is nonsingular and
the Jordan form of $x$ on the $\alpha$ space is the same as the $-\alpha$ space.  Moreover the centralizer
in $\Sp(2m)$ is just the centralizer of $x$ in $\GL(m)$ (on the $\alpha$) space.  In particular, the centralizer
is connected.

We first compute the dimension of the centralizer of an element $x \in \fs$.   As usual, we reduce to the
case where the characteristic polynomial of $x$ is a power of an irreducible polynomial $\phi(x)$ or the
product of two irreducible polynomials  $\phi_i(x)$ where $\phi_2(x) = (-1)^{d}\phi_1(-x)$ where $d$
is the degree of $\phi_i$.   In the first case note
that either $\phi(x) =x$ or $\phi(x)$ has even degree $2d$ (and $\phi(x)=\phi(-x)$).

First consider the case that $\phi(x) = x$ (i.e. $x$ is nilpotent).    Then the orbit of $x$ under $\Sp(2n,k)$ is determined
by the Jordan form of $x$ which can be described as a partition of $n$.
Let $\lambda_1 \ge \lambda_2 \ge \ldots \ge \lambda_r$
(such a partition corresponds to an element of $\fs$ if and only if the multiplicity of the odd pieces in the partition is even).

Recall that $\fs$ as a module either for $G$ or for the Lie algebra is isomorphic to $\Sym^2(V)$ and this is a direct
summand of $V \otimes V$ (which is isomorphic to $V \otimes V^*$ as a module for $x$) since we are assuming
that $q$ is odd.  Since the dimension of
the fixed point space on $V \otimes V$ depends only on the partition and not on the characteristic, the same is true
for the dimension of the fixed point space on $\Sym^2(V)$.   This is well known (for example, this can be computed
in characteristic $0$ using the decomposition of $\Sym^2(V)$ as an $\SL(2)$-module). This proves the first expression
in the following lemma.

\begin{lemma} \label{nilpdimsp}  Let $x \in \fs$ be nilpotent with Jordan block sizes $\lambda_1 \ge \lambda_2 \ge \ldots \ge \lambda_r$
on the natural module.  Then the dimension of the centralizer of $x$ in $\fs$ is equal to both:
\begin{enumerate}
\item
$$
\sum_{j < i}  \lambda_i   + \sum_i  \lceil \lambda_i/2 \rceil,
$$
and
\item \[ \sum_i (\lambda_i')^2/2 + o(\lambda)/2, \] where $o(\lambda)$ is the number of odd parts of $\lambda$.
\end{enumerate}
\end{lemma}

\begin{proof} The expression in part 1 of the lemma is equal to
\begin{eqnarray*}
\sum_{j<i} \lambda_i + \sum_i \lambda_i/2 + o(\lambda)/2 & = &
\sum_i \lambda_i (i-1) + |\lambda|/2 + o(\lambda)/2 \\
& = & \sum_i {\lambda_i' \choose 2} + |\lambda|/2 + o(\lambda)/2 \\
& = &  \sum_i (\lambda_i')^2/2 + o(\lambda)/2.
\end{eqnarray*}
\end{proof}

Now consider the case that $\phi(x) \ne x$.

Then the class of the element is determined by $\phi(x)$ (or the $\phi_i(x)$) and a partition of $n/d$.

To compute the dimension of the centralizer of $x$ in $\fs$, we can work over the algebraic closure
(or any extension field) and so reduce to the case where the characteristic polynomial is a power
of $\phi_1(x) \phi_2(x)$.   Then we can write $V = V_1 \oplus V_2$ where each $V_i$ is a maximal totally
singular subspace, $V_i$ is invariant under $x$ and the characteristic polynomial of $x$ on $V_i$
is a power of $\phi_i(x)$.   Then we see that the centralizer of $x$ in $\fs$ is isomorphic to the centralizer of
$x_i$ (the restriction of $x$ to $V_i)$ in the algebra of linear transformations of $V_i$ and so
the result for $\GL$ gives Theorem \ref{countcentSp} below.

Recall that Jordan canonical form associates to each monic, irreducible polynomial $\phi$ over
$\mathbb{F}_q$ a partition $\lambda_{\phi}$ such that
\[ \sum_{\phi} d(\phi) |\lambda_{\phi}| = 2n .\] This data arises from an element of $\fs$ if and
only if:

\begin{enumerate}
\item The partition $\lambda_x$ has all odd parts occur with even multiplicity
\item For every $\phi$, the partition $\lambda_{\phi}$ is equal to the partition $\lambda_{\overline{\phi}}$, where
$\overline{\phi}$ is defined as $(-1)^{d(\phi)} \phi(-x)$.
\end{enumerate}

Thus, we have proved:

\begin{theorem} \label{countcentSp} If $\{ \lambda_{\phi} \}$ is the data corresponding to an element of $\fs$, then the
dimension of its centralizer in $\fs$ is equal to
\[ \sum_i (\lambda_{x,i}')^2/2 + o(\lambda_x)/2 + \sum_{\phi=\overline{\phi} \atop \phi \neq x} d(\phi) \sum_i (\lambda_{\phi,i}')^2/2
+ \sum_{\{ \phi,\overline{\phi} \} \atop \phi \neq \overline{\phi}} d(\phi) \sum_i (\lambda_{\phi,i}')^2 .\]
\end{theorem}

Finally, we need to count how many elements there are in $\fs$ with a given canonical form.   This amounts to computing
the centralizers in $Sp$.   As in the discussion for centralizers in $\fs$, we can reduce to the three cases as above.  It is a
bit more involved as we now need the order of the centralizer rather than just the dimension.

For nilpotent elements, the centralizer is not usually connected in the algebraic group and so the orbits
break up over the finite field.  However, since the centralizer dimension in the Lie algebra depends only
the orbit over $k$, we only need to know the following.

\begin{lemma} \label{nilcen} The number of nilpotent elements in $\fs$  corresponding to a partition
$\lambda_1 \ge \lambda_2 \ge \ldots \ge \lambda_r$ is
\[ \frac{|\Sp(2n,q)|}{q^{\sum_i (\lambda_i')^2/2 + o(\lambda)/2} \prod_i (1-1/q^2)
\cdots (1-q^{2 \lfloor m_i(\lambda)/2 \rfloor})}  \]
\end{lemma}

\begin{proof}  Since we are in good characteristic, the number of nilpotent elements with a given
Jordan form is the same as the number of unipotent elements in $\Sp(2n,q)$ with the same distribution
of Jordan blocks. This number is determined in \cite{F}.
\end{proof}

For the elements in $\fs$ that do not have $0$ as an eigenvalue, we have seen that the centralizers in the
algebraic group are connected, whence the $\Sp(2n,k)$-orbits intersect $\fs(2n,q)$ are just the $\Sp(2n,q)$
orbits.  So we just need to compute the centralizers in $\Sp(2n,q)$.

The next result gives the number of elements in $\fs(2n,q)$ with a given Jordan canonical
form.

\begin{theorem} \label{countJordSp} If $\{ \lambda_{\phi} \}$ is the data corresponding to an element $x$ of $\fs$, then the number
of elements of $\fs$ corresponding to this data is equal to $|\Sp(2n,q)|$ multiplied by:
\begin{eqnarray*}
& & \frac{1}{q^{\sum_i (\lambda_{x,i}')^2/2 + o(\lambda_x)/2} \prod_i (1-1/q^2) \cdots (1-1/q^{2 \lfloor m_i(\lambda_x)/2 \rfloor})} \\
& & \cdot \prod_{\phi = \overline{\phi} \atop \phi \neq x} \frac{1}{q^{d(\phi) \sum_i (\lambda_{\phi,i}')^2/2}
\prod_i (1+1/q^{d(\phi)/2}) \cdots (1 - (-1)^{m_i(\lambda_{\phi})}/q^{m_i(\lambda_{\phi}) d(\phi)/2}) } \\
& & \cdot \prod_{\{ \phi,\overline{\phi} \} \atop \phi \neq \overline{\phi}} \frac{1}{q^{d(\phi) \sum_i (\lambda_{\phi,i}')^2} \prod_i
(1-1/q^{d(\phi)}) \cdots (1-1/q^{m_i(\lambda_{\phi}) d(\phi)})}
\end{eqnarray*}
\end{theorem}

\begin{proof}
By the discussion before Lemma \ref{nilpdimsp} and Lemma \ref{nilcen}  it suffices to reduce to the case where  $\phi \ne x$
and $\phi$ is either irreducible (and $\phi(x)=\phi(-x)$)  or $\phi=\phi_1\phi_2$ where $\phi_1$ is  irreducible
and $\phi_2$ is the dual polynomial of $\phi_1$  (i.e. its roots are the negatives of the roots of $\phi_1$).

In the latter case, $V$ decomposes as a direct sum of two totally isotropic spaces $V_1 \oplus V_2$ where
the element has characteristic polynomial a power of $\phi_i$ on $V_i$ and the action on $V_2$ is the dual
of the action on $V_1$.   Thus, the centralizer of $x$  is isomorphic to the centralizer in $\GL(V_1)$ of $x$ restricted to $V_1$
and the result follows by Lemma \ref{comwith}.

So it remains to consider the case that $\phi$ is irreducible.   The unipotent radical of the centralizer has order
$q^{\dim Q}$ where $Q$ is the unipotent radical of the centralizer in the algebraic group (and so can be computed
as above).   So we only have to identify the reductive part.      Modulo the unipotent radical, we see that we can
reduce to the case where all parts of the partition have the same size $e$  (this can be seen in the algebraic group).
Let $m$ denote the this multiplicity  and let $2d$ be the degree of $\phi$.

If we pass to a quadratic extension of our finite field, we are in the previous case and so the reductive
part of the centralizer is $\GL(m, q^{2d})$.   Arguing as we did in the unitary case, we see that the reductive
part of the centralizer is $\GU(m,q^d)$.
\end{proof}

Recall that $\overline{\phi}$ is defined as $(-1)^{d(\phi)} \phi(-x)$.
Note that if a monic irreducible polynomial $\phi \neq x$ satisfies $\overline{\phi}=\phi$, then $\phi$ has even degree.
We let $\overline{N}(2d,q)$ denote the number of monic irreducible polynomials $\phi$ of degree $2d$ such that
$\overline{\phi}=\phi$. It is also helpful to let $\overline{M}(d,q)$ denote the number of unordered pairs
$\{ \phi, \overline{\phi} \}$ of monic, irreducible, degree $d$ polynomials such that $\phi \neq \overline{\phi}$.

The following enumerative lemma will be helpful. It is a symplectic analog of Lemma \ref{poly}.

\begin{lemma} \label{countpolySp}
\begin{enumerate}

\item \[ \prod_{d \geq 1} (1-u^d)^{- \overline{N}(2d,q)} \prod_{d \geq 1} (1-u^d)^{- \overline{M}(d,q)} = \frac{1-u}{1-qu} .\]

\item \[ \prod_{d \geq 1} (1+u^d)^{- \overline{N}(2d,q)} \prod_{d \geq 1} (1-u^d)^{- \overline{M}(d,q)} = 1 .\]

\end{enumerate}
\end{lemma}

\begin{proof} To prove the first assertion, note that any monic polynomial (not necessarily irreducible) which is invariant
under the involution $f(x) \mapsto (-1)^{d(f)} f(-x)$ factors uniquely into irreducibles as a product of powers of $\phi$ where
$\overline{\phi}=\phi$ and of powers of $\phi \overline{\phi}$ where $\phi \neq \overline{\phi}$. Hence the coefficient
of $u^n$ in \[ (1-u)^{-1} \prod_{d \geq 1} (1-u^d)^{- \overline{N}(2d,q)} \prod_{d \geq 1} (1-u^d)^{- \overline{M}(d,q)} \] is equal
to the number of monic degree $2n$ polynomials $f$ such that $f(x)=f(-x)$. This is equal to $q^n$, which is the
coefficient of $u^n$ in $\frac{1}{1-qu}$, proving the first part of the theorem.

To prove the second part of the theorem, we claim that
\[ (1-u)^{-1} \prod_{d \geq 1} (1-u^{2d})^{- \overline{N}(2d,q)} \prod_{d \geq 1} (1-u^d)^{- 2 \overline{M}(d,q)} = \frac{1}{1-qu} .\]
Indeed, the claimed equation follows since the coefficient of $u^n$ in the left hand side counts the total number of monic degree $n$
polynomials, as does the coefficient of $u^n$ in $1/(1-qu)$. Dividing the claimed equation by the first assertion of the theorem
immediately proves the second assertion of the theorem.
\end{proof}

Let $S_0=1$, and for $n \geq 1$, let $S_n$ be the number of ordered pairs of commuting elements of the Lie
algebra of $\Sp(2n,q)$. Armed with the above results, we can now derive a generating function for the sequence $S_n$.

\begin{theorem} \label{countpairsSp}
\[ \sum_{n \geq 0} \frac{S_n u^n}{|\Sp(2n,q)|} = \frac{\prod_{i \geq 1} (1+u^i)}
{\prod_{i \geq 1} \prod_{l \geq 0} (1- u^i/q^{2l-1})} .\]
\end{theorem}

\begin{proof} It is immediate from Theorems \ref{countcentSp} and \ref{countJordSp} that $S_n$ is equal to
$|\Sp(2n,q)|$ multiplied by the coefficient of $u^{2n}$ in $ABC$ where
\[ A = \sum_{\lambda \atop i \ odd \implies m_i(\lambda) \ even} \frac{u^{|\lambda|}}{\prod_i (1-1/q^2) (1-1/q^4) \cdots (1-1/q^{2 \lfloor m_i(\lambda)/2 \rfloor})} \]
\[ B = \prod_{\phi = \overline{\phi} \atop \phi \neq x} \sum_{\lambda} \frac{u^{d(\phi) |\lambda|}}{\prod_i (1+1/q^{d(\phi)/2}) \cdots
(1-(-1)^{m_i(\lambda)}/q^{m_i(\lambda) d(\phi)/2})} \]
\[ C = \prod_{ \{\phi,\overline{\phi} \} \atop \phi \neq \overline{\phi}} \sum_{\lambda} \frac{u^{2 d(\phi) |\lambda|}} {{\prod_i (1-1/q^{d(\phi)}) \cdots (1-1/q^{m_i(\lambda) d(\phi)})}} .\] Note that in $A$, the sum is over all partitions such that the odd parts occur with even multiplicity, and that in
$B$ and $C$, the sum is over all partitions.

Note that $A=(A1)(A2)$ where

\[ A1 = \prod_{i \ odd} \sum_{m \ even} \frac{u^{im}}{(1-1/q^2) (1-1/q^4) \cdots (1-1/q^m)} \] and $A2$ is equal to
\[ \prod_{i \ even \atop i \geq 2} \left[ \sum_{m \ even} \frac{u^{im}}{(1-1/q^2) \cdots (1-1/q^{m})} +
\sum_{m \ odd} \frac{u^{im}}{(1-1/q^2) \cdots (1-1/q^{m-1})} \right] \]

Applying Lemma \ref{euler} gives that

\[ A1 = \prod_{i \ odd} \prod_{l \geq 0} \frac{1}{(1-u^{2i}/q^{2l})} \]
\[ A2 = \prod_{i \ even \atop i \geq 2} (1+u^i) \prod_{l \geq 0} \frac{1}{(1-u^{2i}/q^{2l})}. \]

Next observe that

\begin{eqnarray*}
B & = & \prod_{\phi = \overline{\phi} \atop \phi \neq x} \prod_{i \geq 1} \sum_{m \geq 0} \frac{u^{im d(\phi)}}{(1+1/q^{d(\phi)/2}) \cdots
(1-(-1)^m/q^{m d(\phi)/2})} \\
& = & \prod_{\phi = \overline{\phi} \atop \phi \neq x} \prod_{i \geq 1} \prod_{l \geq 0} \frac{1}{(1-(-1)^l u^{i d(\phi)}/q^{l d(\phi)/2})}
\end{eqnarray*} where the second equality used Lemma \ref{euler}.

Similarly,

\begin{eqnarray*}
C & = & \prod_{ \{\phi,\overline{\phi} \} \atop \phi \neq \overline{\phi}} \prod_{i \geq 1} \sum_{m \geq 0}
\frac{u^{2im d(\phi)}}{(1-1/q^{d(\phi)}) \cdots (1-1/q^{m d(\phi)})} \\
& = & \prod_{ \{\phi,\overline{\phi} \} \atop \phi \neq \overline{\phi}} \prod_{i \geq 1} \prod_{l \geq 0}
\frac{1}{(1-u^{2i d(\phi)}/q^{l d(\phi)})}.
\end{eqnarray*}

Switching the order of the products in $B$ and $C$, and then applying both parts of Lemma \ref{countpolySp}
gives that

\begin{eqnarray*}
BC & = & \prod_{l \geq 0} \prod_{i \geq 1} \left[ \prod_{d \geq 1} (1-(-1)^l u^{2id}/q^{ld})^{- \overline{N}(2d,q)}
\prod_{d \geq 1} (1-u^{2id}/q^{ld})^{- \overline{M}(d,q)} \right] \\
& = & \prod_{l \ even \atop l \geq 0} \prod_{i \geq 1} \frac{(1-u^{2i}/q^l)}{(1-u^{2i}/q^{l-1})}.
\end{eqnarray*}

Hence \[ ABC = (A1)(A2)(BC) = \frac{\prod_{i \geq 1} (1+u^{2i})}
{\prod_{i \geq 1} \prod_{l \geq 0} (1-u^{2i}/q^{2l-1})}, \] which proves the theorem.
\end{proof}

Next we use Theorem \ref{countpairsSp} to obtain asymptotic information about the number of commuting pairs
in the Lie algebra of $\Sp(2n,q)$.

\begin{theorem} For $q$ fixed,
\[
\lim_{n \rightarrow \infty} \frac{S_n}{q^{2n^2+2n}}  = \prod_{i \geq 1}
(1+1/q^i) (1-1/q^i)^{- \lfloor (i+1)/2 \rfloor} .\]
\end{theorem}

\begin{proof} We know from Theorem \ref{countpairsSp} that
\begin{eqnarray*}
S_n & = & |\Sp(2n,q)| [u^n] \frac{\prod_{i \geq 1} (1+u^i)}{\prod_{i \geq 1} \prod_{l \geq 0} (1-u^i/q^{2l-1})} \\
& = & |\Sp(2n,q)| q^n [u^n] \frac{\prod_{i \geq 1} (1+u^i/q^i)}{\prod_{i \geq 1} \prod_{l \geq 0} (1-u^i/q^{2l+i-1})} \\
& = & q^{2n^2+2n} (1-1/q^2) \cdots (1-1/q^{2n}) \\
& & \cdot [u^n] \frac{\prod_{i \geq 1} (1+u^i/q^i)}{\prod_{i \geq 1} \prod_{l \geq 0}
(1-u^i/q^{2l+i-1})}.
\end{eqnarray*}

Thus \[ \lim_{n \rightarrow \infty} \frac{S_n}{q^{2n^2+2n}} = \prod_{j \geq 1} (1-1/q^{2j}) \lim_{n \rightarrow \infty}
[u^n] \frac{\prod_{i \geq 1} (1+u^i/q^i)}{\prod_{i \geq 1} \prod_{l \geq 0}
(1-u^i/q^{2l+i-1})} .\]

Since the $i=1,l=0$ case of $1/(1-u^i/q^{2l+i-1})$ is equal to $1/(1-u)$, it follows from Lemma \ref{tay} and
elementary manipulations that \[ \prod_{j \geq 1} (1-1/q^{2j}) \lim_{n \rightarrow \infty}
[u^n] \frac{\prod_{i \geq 1} (1+u^i/q^i)}{\prod_{i \geq 1} \prod_{l \geq 0}
(1-u^i/q^{2l+i-1})} \] is equal to
\[ \prod_{i \geq 1}
(1+1/q^i) (1-1/q^i)^{- \lfloor (i+1)/2 \rfloor}    .\]
\end{proof}

{\it Remark:}
Let $\Com(G)$ denote the number of ordered pairs of commuting elements of a finite group $G$. As
mentioned earlier, $\Com(G)$ is equal to $k(G)|G|$, where $k(G)$ is the number of conjugacy classes of $G$. From
\cite{FG}, since $q$ is odd, one has that
\[ \lim_{n \rightarrow \infty} \frac{k(\Sp(2n,q))}{q^n} = \prod_{i \geq 1} \frac{(1+1/q^i)^4}{(1-1/q^i)}.\] Thus,
\begin{eqnarray*}
\lim_{n \rightarrow \infty} \frac{\Com(\Sp(2n,q))}{q^{2n^2+2n}} & = & \prod_{i \geq 1} \frac{(1+1/q^i)^4}{(1-1/q^i)}
\lim_{n \rightarrow \infty} \frac{|\Sp(2n,q)|}{q^{2n^2+n}} \\
& = & \frac{\prod_{i \geq 1} (1+1/q^i)^4}{\prod_{i \ odd}
(1-1/q^i)}.
\end{eqnarray*}

Our next goal is to count the number of nilpotent commuting pairs in $\fs$.

We first require a result of Richardson \cite[Lemma 6.6]{R1} (which can be proved quite easily in our special case).   Note that for
$G=\Sp(2n,k)$ all characteristics other than $2$ are very good.

\begin{lemma} \label{lem:richardson} \rm{(\cite{R1})}  Let $G$ be a simple algebraic group
over an algebraically closed field $k$.
Assume that the characteristic is very good for $G$.  Let $\frak{g}$ be the Lie algebra of $G$.   If $x \in \frak{g}$,
then $\{y \in \frak{g}|[x,y]=0\}$ is the Lie algebra of $C_G(x)$.
\end{lemma}

\begin{lemma} \label{nilcen2} If an element $x$ of $\fs$ is nilpotent of type $\lambda$, then its centralizer in $\fs$ has size
\[ q^{\sum_i (\lambda_i')^2/2 + o(\lambda)/2 - \sum_i \lfloor m_i(\lambda)/2 \rfloor} .\]
\end{lemma}

\begin{proof}   Note that we are in very good characteristic.   As we have already seen, the dimension of the centralizer
in the group
of a nilpotent element $x$ is  $\sum_i (\lambda_i')^2/2 + o(\lambda)/2$.   It follows by \cite{LS}  that the rank of the centralizer is
$\sum_i \lfloor m_i(\lambda)/2 \rfloor$.   By the result of Richardson above, the same is true for the centralizer
of  $x$ in $\fs$.  The result now follows by
 Lemma \ref{lem:nilpotent},
\end{proof}

This leads to the following theorem. In its statement, we let $NS_0=1$, and let $NS_n$ denote the number
of commuting ordered pairs of nilpotent elements in the Lie algebra of $\Sp(2n,q)$.

\begin{theorem} \label{nilpcomSp}
\[ \sum_{n \geq 0} \frac{NS_n u^n}{|\Sp(2n,q)|} = \frac{\prod_{i \geq 1} (1+u^i)}{\prod_{i \geq 1} \prod_{l \geq 0}
(1-u^i/q^{2l+1})} .\]
\end{theorem}

\begin{proof} It follows from Lemmas \ref{nilcen} and \ref{nilcen2} that $NS_{n}$ is equal to
\[ |\Sp(2n,q)| \sum_{|\lambda|=2n \atop i \ odd \ \implies m_i(\lambda) \ even} \frac{1}{\prod_i (1-1/q^2) \cdots (1-1/q^{2 \lfloor m_i(\lambda)/2 \rfloor})
q^{\lfloor m_i(\lambda)/2 \rfloor}} .\] Letting $[u^n] f(u)$ denote the coefficient of $u^n$ in a power series $f(u)$, it follows that
\[ NS_n = |\Sp(2n,q)| [u^{2n}] AB \] where
\[ A = \prod_{i \ odd} \sum_{m \ even} \frac{u^{im}}{(1-1/q^2) \cdots (1-1/q^m) q^{m/2}} \]
\begin{eqnarray*}
 B & = & \prod_{i \ even \atop i \geq 2} \sum_{m} \frac{u^{im}}{(1-1/q^2) \cdots (1-1/q^{2 \lfloor m/2 \rfloor}) q^{\lfloor m/2 \rfloor}} \\
& = & \prod_{i \ even \atop i \geq 2} (1+u^i) \sum_{m \ even} \frac{u^{im}}{(1-1/q^2) \cdots (1-1/q^m) q^{m/2}}.
\end{eqnarray*}

It follows from Lemma \ref{euler} that \[ AB = \frac{\prod_{i \geq 1} (1+u^{2i})}{\prod_{i \geq 1} \prod_{l \geq 0}
(1-u^{2i}/q^{2l+1})}, \] proving the theorem.
\end{proof}

\section{Acknowledgements} Fulman was partially supported by Simons Foundation Grant 400528.
Guralnick was partially supported by NSF grant DMS-1600056.
The authors thank Matt Young for asking about analogs of the Feit-Fine theorem for
other Lie algebras.

\end{document}